\documentclass[10pt]{article}
\usepackage{pgfplots}
\usepackage{mathrsfs}
\usetikzlibrary{arrows}
\usepackage{amscd,amsmath,amsfonts,amssymb,amsthm,cite,mathrsfs,color} 
\usepackage{dsfont} \usepackage{leftidx}
\usepackage{tikz} \usepackage{enumerate}
\usepackage{appendix}

\usepackage{hyperref}

\usepackage{geometry}

\usepackage{graphicx} \usepackage{xcolor}

\numberwithin{equation}{section} 

\usepackage{amsthm}

\theoremstyle{plain}
\newtheorem{theorem}{Theorem}[section]
\newtheorem{lemma}[theorem]{Lemma}

\newtheorem{proposition}[theorem]{Proposition}

\theoremstyle{definition}
\newtheorem{definition}[theorem] 
{Definition}

\newtheorem{remark}[theorem]{Remark}

\newtheorem*{acknow}{Acknowledgments}
\begin{document}
	
	\title{\bf{Local 
			limit theorems for random walks on a  large discrete torus  }}
	
	\author{
		Yandong Gu\footnotemark[1],   ~Dang-Zheng Liu\footnotemark[1] ~ 
	}
	\renewcommand{\thefootnote}{\fnsymbol{footnote}}
	\footnotetext[1]{School of Mathematical Sciences, University of Science and Technology of China, Hefei 230026, P.R.~China. E-mail: gd27@mail.ustc.edu.cn, dzliu@ustc.edu.cn
	}

	\maketitle
	\begin{abstract}
		Inspired by the study of edge statistics of random band matrices, we  investigate random walks on  large $d$-dimensional periodic  lattices, whose   transition matrices are determined  by   discretized   density functions.  Under certain moment assumption  on the density,  we prove local limit theorems for random walks  in three different regimes  according to   the  bandwidth parameter, random walk length  and torus size.

		
	\end{abstract}
	
	\section{Introduction}
	
	\subsection{Motivation}

	Given a  $d$-dimensional lattice
	\begin{equation} \label{latt}
		\Lambda_{L,d}=\Big( \big(-\frac{1}{2}L,\frac{1}{2}L\big] 
		\bigcap \mathbb{Z}\Big)^d, 
	\end{equation}
	by defining  a  canonical representative
	of $x\in  \Lambda_{L,d} $ through
	\begin{equation} [x]_L: = \Big(\big(x_1+L\mathbb{Z}\big)\times \cdots \times \big(x_d+L\mathbb{Z}\big)\Big)\bigcap  \Lambda_{L,d}, \end{equation} 
	we  introduce   a periodic  distance  on $\Lambda_{L,d}$  for any  $l^p$ ($p\geq1$) norm   
	\begin{equation} \|x-y\|_{p}: = \|[x-y]_L\|_{p}, \end{equation} 
	and  set 
	\begin{equation}N:= |\Lambda_{L,d}|=L^d.
	\end{equation}
	So we can define a random band matrix on $\Lambda_{L,d}$, one of the most prominent  matrix models in Random Matrix Theory.  	A symmetric or Hermitian matrix $H=(H_{xy})_{x,y \in \Lambda_{L,d}}$ with $H_{xy}= \sigma_{xy}A_{xy}$ is  said to be   an unimodular random band matrix with bandwidth $W$ and variance profile
	\begin{equation}
		\sigma_{xy}^2=\frac{1}{\Omega_W}f\Big(\frac{[x-y]_L}{W}\Big),\quad \Omega_W=\sum_{0\neq x \in \Lambda_{L,d}}f\big(\frac{x}{W}\big),
	\end{equation}
	if  the following three  assumptions hold:
	(1) \{$A_{xy}$\} are independent up to Hermitian  symmetry;
	(2) \{$A_{xy}$\} are uniform on \{+1,-1\} ($\beta=1$) or \{$e^{i \theta},\quad \theta \in [0,2\pi) $\}($\beta=2$);
	(3) $f(x)$ is a symmetric probability density function ($f(-x)=f(x)$).

	Local statistics   for extreme eigenvalues of random band matrices  is one of the core problems in Random Matrix Theory.  An effective approach of studying it  is to use the modified moment method \cite{lz23,sodin10,soshnikov99}, in which 
	a crucial step  is to  obtain  local limit theorems and upper bounds of transition kernels for random walks on  large $d$-dimensional periodic  lattices.  
	More specifically,    random walks on the torus  can be defined as  follows.  
	
	\begin{definition}(Circulant random walk)\label{def-CRM}
		Given    a symmetric   density function  $f(x)$   on $\mathbb{R}^d$,  $S_n$ is said to be a circulant random walk 
		with   profile $f$ and  bandwidth parameter $W>0$  on the  torus  $\Lambda_{L,d}$      if it 
		is  a  Markov
		chain with transition matrix  
		\begin{equation}\label{defp1} 
			p_1(x,y)=\frac{1}{\Omega_W}f\Big(\frac{[y-x]_L}{W}\Big), \quad x\neq y \in \Lambda_{L,d},  \end{equation}
		where \begin{equation}
			\Omega_W=\sum_{0\neq x\in \Lambda_{L,d}}f\left(\frac{x}{W}\right).
		\end{equation}
		The $n$-step  transition probability between $0$ and $x$ in  $\Lambda_{L,d}$ is denoted by $p_n(x):=p_n(0,x)$.\end{definition}

	It is worth  stressing  that  $S_n$ is not only  temporally homogeneous  but also  spatially homogeneous. 
	Very interesting examples  for  densities    include  uniform distributions on hypercubes  and spheres   (cut-off functions) with $ r\in (0,\infty)$
	\begin{equation}
		f(x)=\prod_{i=1}^d \frac{1}{2r}1_{\{|x_i|\leq r\}},  
	\end{equation}
	\begin{equation}
		f(x)= \frac{\Gamma(\frac{d}{2}+1)}{\pi^{\frac{d}{2}}r^d}1_{\{\|x\|_2\leq r\}}, 
	\end{equation}
	and the Gaussian density 
	\begin{equation}
		f(x)= \frac{1}{(\sqrt{2\pi})^ d} e^{-\frac{1}{2}\|x\|_{2}^2} , \quad x\in \mathbb{R}^d.
	\end{equation}
	
	Random walk   is one of the most basic
	topics in probability theory and  has  fundamental  applications in   various fields of mathematics and physics \cite{ll10}.  Particularly, it may play a   significant  role in the study of  spectral properties  and eigenvector  statistics  in  Random Matrix Theory,   which is reflected in  using  moment method and moment estimate; see e.g.  \cite{Ek11b,EK11,fs10,lz23,sodin10 ,soshnikov99}.

	Our  circulant random walk in Definition \ref{def-CRM} is related to   geometric structure  of the underlying graphs or groups.  There are a lot of works in the literature,   which  are devoted to   limit theorems  and heat kernel bounds for   random walks  on graphs and groups,  see  e.g.  \cite{bar17,bq16,cs08,cs15,dp88,wo00} and references therein.    Particularly,  see  \cite{chen16,ckw20} for  recent results about upper or lower bounds of  heat kernels.
	However,  as far as we know,   there are very  few  works related to  limit theorems and heat kernel bound   for the circulant random walk with a large  parameter  $W$.

	{\bf Notation.}  $f(x)\asymp g(x)$ means  that there are two positive constants $C_1$ and $C_2$ such that 
	$C_1 g(x)\leq f(x) \leq C_2 g(x)$ for any $ x\in \mathbb{R}^d$.  For real numbers $a,b$,  $a  \wedge  b:=\min\{a,b\}$  and  $a \vee b:=\max\{a,b\}$.
	
	\subsection{Main results}
	
	In order to state the main results, we need the well-known  theta functions  characterized by the $\alpha$-stable densities. \begin{definition}($\alpha$-stable theta functions)\label{def-tf}
		
		(i) For $\alpha \in(0,2)$, 
		\begin{equation}
			\theta_{\alpha}(x,\tau)
			=\sum_{k\in \mathbb{Z}^d} f_{\alpha}(x+k,\tau)
		\end{equation}
		where  $f_{\alpha}(x,\tau):=\tau^{-\frac{d}{\alpha}}f_{\alpha}(\tau^{-\frac{1}{\alpha}}x)$, and $f_{\alpha}(x)$ is the density function of an isotropic multivariate $\alpha$-stable distribution with characteristic function 
		$\phi(t)=\exp\{-C_{\alpha}\|t\|_2^{\alpha}\}$ with $C_{\alpha}>0$.
		In fact, one can define $f_{\alpha}$ for any $C_{\alpha}>0$, and for convenience, here we take the fixed $C_{\alpha}$ from
		\eqref{plorder}.
		
		(ii) Given a positive definite matrix $\Gamma$ of size $d$,  the multidimensional Jacobi theta function is defined as
		\begin{equation}\label{theta}
			\theta(z,\Gamma)=
			\sum_{k\in \mathbb{Z}^d} \exp \{-\pi 
			<k\Gamma,k> +2\pi i<k,z> \},
		\end{equation}	where  $<\cdot,\cdot>$ denotes the inner product.
	\end{definition}


	\begin{theorem}\label{pllt} For the random walk given  in  Definition \ref{def-CRM}, 
		assume    that $f$ is continuous   at $x_0$ and $f(x_0)>0$ for  some fixed  $x_0\in [0,\frac{L}{2W})^d$, and also  
		\begin{equation}\label{pldef}
			f(x)\asymp  \big(\|x\|_2 \vee1 \big)^{-d-\alpha}, 
		\end{equation}
		where $\alpha \in  (0,2)$. 
		As $L\to \infty $ and $W\to\infty$, the following three types of local  limits  hold.
		\begin{itemize}
			\item [(I)](Subcritical regime)  When $\log W\ll n \ll (L/W)^{\alpha}$,  
			\begin{equation}
				p_n(x)=	f_{\alpha}(x,nW^{\alpha})(1+o(1)).
			\end{equation}
			\item[(II)] (Critical regime)
			When $ n \sim (L/W)^{\alpha}$ and $n\gg \log W$, 
			\begin{equation}
				p_n(x)	=\frac{1}{N}  {\theta}_{\alpha}\Big(\frac{x}{L},   \frac{n W^{\alpha}}{L^{\alpha}} \Big)\big(1+o(1)\big),
			\end{equation}

			\item[(III)] (Supercritical regime) When  $n\gg (L/W)^{\alpha}$ and $n\gg \log L$,
			\begin{equation}	
				p_n(x)	=\frac{1}{N}\big(1+o(1)\big).
			\end{equation}
		\end{itemize}	
	\end{theorem}
	\begin{proposition}\label{p-uub}
		Without any further restrictions on $W$, $n$ and  $x$, there is a constant $C>0$ such that
		\begin{equation}\label{pluub}
			p_n(x)\leq C\left(\frac{1}{n^{\frac{d}{\alpha}}W^d}\wedge \frac{nW^{\alpha}}{(\|x\|_2 \vee W) ^{d+\alpha}}+\frac{1}{N} \right).
		\end{equation}
	\end{proposition}

	\begin{theorem}
		\label{LLT}
		For the random walk given  in  Definition \ref{def-CRM}, 
		assume   that there is a fixed real number $x_0\in [0,\frac{L}{2W})^d$ such that $f$ is continuous   at $x_0$ and $f(x_0)>0$, and 
		the covariance matrix $\Gamma=(\Gamma_{ij})$ of density $f$ given   by 
		\begin{equation}  
			\Gamma_{ij}= \int_{\mathbb{R}^d} x_i x_j f(x)dx \end{equation} 
		exists and  is positive definite. Set 
		\begin{equation} \label{LCVar}
			(\Gamma_{L})_{ij}
			= \frac{1}{\Omega_W} \sum_{0\neq x\in \Lambda_{L,d}} 
			\frac{x_i x_j}{W^2} f\left(\frac{x}{W}\right),  \ i,j=1, \ldots,L,
		\end{equation}
		and 
		\begin{equation} \label{LVar}\sigma_{L}^2:=\frac{W^2}{L^2} <k\Gamma_L, k>.
		\end{equation}
		As $L \to \infty $ and $W\to\infty$, the following three types of local  limits  hold.
		\begin{itemize}
			\item [(I)](Subcritical regime)  When $\log W\ll n \ll (L/W)^2$ and $\|x\|_2 \ll n^{2/3}W$,
			\begin{equation}
				p_n(x)=	
				\frac{1+o(1)}{ (\sqrt{2\pi nW^2\det \Gamma_L})^{d}}   e^{-\frac{1}{2nW^2} 
					<x\Gamma_L^{-1},x> }.
			\end{equation}
			
			\item[(II)] (Critical regime)
			When $ n \sim (L/W)^2$ and $n\gg \log W$,  and  $\|x\|_2 \ll n^{2/3}W$
			\begin{equation}
				p_n(x)	=\frac{1}{N}  \theta\big(\frac{x}{L}
				,  2\pi \frac{n W^2}{L^2}\Gamma_{L}
				\big)\big(1+o(1)\big).
			\end{equation}	
			
			\item[(III)] (Supercritical regime) When  $n\gg (L/W)^2$ and $n\gg \log L$,
			\begin{equation}	
				p_n(x)	=\frac{1}{N}\big(1+o(1)\big).
			\end{equation}
			
		\end{itemize}	
		In addition,  it is also assumed     that 
		$\int_{\mathbb{R}^d} \|x\|_{2}^3  f(x)dx<\infty$ in Case (I) and Case (II). 
	\end{theorem}
	
	\begin{remark} 
		%
		(i)	
		By combining  local limit theorems and upper bounds  in Theorem \ref{pllt} and Proposition \ref{p-uub}  above with  analytical techniques for Feynman graphs  developed by one author of this  paper and Zou \cite{lz23}, one  can  establish the edge statistics for $d$-dimensional 
		unimodular RBM with $\alpha$ power-law variance profile ($\alpha \in (0,2)$). We  refer the  reader to Theorems 1.2,  1.3,  1.4 and their proofs in \cite{lz23} for more details.

		(ii)
		In dimension  $d=1$,  Diaconis has studied   random walks on the circle with finite bandwidth  in the critical and supercritical regimes  in \cite{dp88}. 
		Also in dimension  $d=1$,
		 for the cut-off variance profile  and for  the large bandwidth  Sodin \cite{sodin10} established local limit theorems in both subcritical and supercritical regimes and stated the upper bound (without a detailed proof  for the latter).   In  general dimension  $d\geq 1$,   for the $\alpha$-stable   variance profile   Liu and Zou \cite{lz23} have established   local limit theorems and  upper bounds.    However, the heat kernel upper bound correspondingly to Theorem \ref{LLT}, except for the Gaussian profile   in  \cite[Proposition 3.4]{lz23}, is not available, so we can not    complete  the edge statistics for $d$-dimensional 
		unimodular RBM with   rapidly  decaying   variance profile. 
		
	\end{remark}

	Our method   is based on the  discrete Fourier transform on Abelian groups \cite{ter99}, however,  we have to  do some   careful  summation estimates because of the existence of the  large bandwidth $W$.  We will prove 
	Theorem \ref{pllt} and Proposition \ref{p-uub} in Section \ref{sect2}, and prove Theorem \ref{LLT} in   Section \ref{sect3}.

	\section{Proofs of Theorem \ref{pllt} and Proposition \ref{p-uub}} \label{sect2}
	
	\subsection{One-step transition probability}
	In order to verify the above theorems, we will make use of (inverse)  discrete Fourier transform on finite  groups.   	Notice   the  recursive  relation for convolutions  
	\begin{equation}
		p_n(x)=p_1*p_{n-1}(x), \  n\ge 1,  \label{convol}
	\end{equation}
	using   the discrete Fourier transform defined   by 
	\begin{equation}
		\hat{p}_n(k)=\sum_{x\in \Lambda_{L,d}} p_n(x)e^{2\pi i\sum _{j=1}^d\frac{k_jxj}{L}}, 
	\end{equation} 
	we obtain from \eqref{convol} that 
	\begin{equation}
		\hat{p}_n(k)=\hat{p}_1(k)\hat{p}_{n-1}(k).
	\end{equation}
	Hence 
	\begin{equation}
		\hat{p}_n(k)=(\hat{p}_1(k))^n.
	\end{equation}
	Again by the inverse 
	discrete   Fourier transform,  we  get an expression  of $p_n(x)$ in terms of $\hat{p}_1(k)$. 
	\begin{equation}
		p_n(x)
		=\frac{1}{N}\sum_{k\in \Lambda_{L,d}}(\hat{p}_1(k))^n  e^{-2\pi i \sum_{j=1}^{d}\frac{k_jx_j}{L }}. \label{IFT}
	\end{equation}
	
	Now we give a key estimate  for  $\hat{p_1}$.
	\begin{lemma}\label{plf}
		For  $\delta>0$ sufficiently small, there exists a constant $C_{\alpha}>0$  and $\rho<1$ such that  
		\begin{equation}\label{plorder}
			1-\hat{p}_1(k)
			=C_{\alpha}\big(W\|k/L\|_2\big)^{\alpha}+ O\Big(\big(W\|k/L\|_2\big)^{2}\Big),
			\quad \forall\, \|k\|_2 \leq \delta \frac{L}{W},
		\end{equation}
		and 
		\begin{equation}\label{plerr}
			\left|\hat{p}_1(k)\right|\leq \rho, \quad \forall\,\|k\|_2>\delta \frac{L}{W}.
		\end{equation}
	\end{lemma}
	\begin{proof}
	We start by establishing  equation \eqref{plorder} for $\|k\|_2 \leq \delta \frac{L}{W}$.
	Based on the decay properties of  function $f$  in  \eqref{pldef}, there exists a constant $l<\infty$ such that 
		\begin{equation}\label{pld}
			f(x)\asymp \big(\|x\|_2\big)^{-d-\alpha},\quad  \|x\|_2 \geq l.
		\end{equation}
To analyse $1-\hat{p}_1(k)$, we  partition it according to the range of $x$:
		\begin{equation}
			\begin{aligned}	1-\hat{p}_1(k)&=\frac{1}{\Omega_W}\sum_{x\in \Lambda_{L,d}} f\Big(\frac{x}{W}\Big)\big(1-e^{2\pi i\sum _{j=1}^d\frac{k_jxj}{L}}\big)\\
				&=\frac{1}{\Omega_W}\sum_{x\in \Lambda_{L,d}} f\Big(\frac{x}{W}\Big)\big(1-\cos\big(2\pi\sum _{j=1}^d\frac{k_jxj}{L}\big)\big)\\
				&=:
				S_1+S_2+S_3,	
			\end{aligned}
		\end{equation}
		where 
		\begin{equation}
			S_1=\frac{1}{\Omega_W}\sum_{\|x\|_2\leq lW}f\Big(\frac{x}{W}\Big)\big(1-\cos\big(2\pi\sum _{j=1}^d\frac{k_jxj}{L}\big)\big),
		\end{equation}
		\begin{equation}
			S_2=\frac{1}{\Omega_W}\sum_{lW\leq\|x\|_2\leq \frac{L}{\|k\|_2}}f\Big(\frac{x}{W}\Big)\big(1-\cos\big(2\pi\sum _{j=1}^d\frac{k_jxj}{L}\big)\big),
		\end{equation}
	and 
		\begin{equation}
			S_3=\frac{1}{\Omega_W}\sum_{\|x\|_2\geq\frac{L}{\|k\|_2}}f\Big(\frac{x}{W}\big)\Big(1-\cos\big(2\pi\sum _{j=1}^d\frac{k_jxj}{L}\big)\big).
		\end{equation}
		
		Using  the approximation $1-cos(kx)\asymp k^2x^2$ if $|x|\leq \frac{\pi}{2|k|}$, we can derive the following estimates 
		\begin{equation}
			S_1\leq O(\|k/L\|_2^2)\frac{1}{\Omega_W}\sum_{\|x\|_2\leq lW}f\Big(\frac{x}{W}\Big)\|x\|_2^2=O\big(W^2\|k/L\|_2^2\big),
		\end{equation}
		\begin{equation}
			S_2\asymp O\big(W^{\alpha}\|k/L\|_2^2\big)\sum_{lW\leq\|x\|_2\leq \frac{L}{\|k\|_2}}\|x\|_2^{-d-\alpha+2}=O\big(W^{\alpha}\|k/L\|_2^{\alpha}\big).
		\end{equation}
	Similarly, using   the fact that $1-cos(kx)\leq 2$ we have 
		\begin{equation}
			S_3\leq O\big(W^{\alpha}\big)\sum_{\|x\|_2\geq\frac{L}{\|k\|_2}}\|x\|_2^{-d-\alpha}=O\big(W^{\alpha}\|k/L\|_2^{\alpha}\big).
		\end{equation}
	It is easy to see that $S_2$ dominates the other terms from above  estimates, thereby completing the proof of equation \eqref{plorder}.\\
		
		Next, we will prove \eqref{plerr} under the condition $\|k\|_2 > \frac{L}{W}\delta$.
		Since $f$ is continuous   at $x_0$ and $f(x_0)>0$ for  some fixed  $x_0\in [0,\frac{L}{2W})^d$, we can choose a positive constant $\epsilon_0 $ such that 
		\begin{equation}
			f(\frac{x}{W})>\frac{1}{2}f(x_0), \quad \mathrm{if}\ \lVert x-x_0 W\rVert_{\infty}\leq \epsilon_0 W,
		\end{equation}
This allows us to partition $\hat{p}_1(k)$, resulting in the following estimate:
		\begin{align}
			\left|\hat{p}_1(k)\right|&\leq
			\frac{1}{\Omega_W}\left(\sum_{0\neq x, \lVert x-x_0 W\rVert_{\infty}\leq \epsilon_0 W}\Big(f\Big(\frac{x}{W}\Big)-\frac{f(x_0)}{2}\Big)+\sum_{0\neq x, \lVert x-x_0 W\rVert_{\infty}>\epsilon_0 W}f\left(\frac{x}{W}\right)\right)\nonumber \\
			&+\frac{1}{\Omega_W}\sum_{ 0\neq x, \lVert x-x_0 W\rVert_{\infty}\leq \epsilon_0 W}\frac{f(x_0)}{2}e^{2\pi i \sum_{j=1}^d\frac{k_jx_j}{L}}.
		\end{align}
		Without loss of generality, we just consider $x_0 =0$ because of  the periodic structure.      In this case, we have 
		\begin{align}
			\left|\hat{p}_1(k)\right|&\leq 1- 
			\frac{1}{\Omega_W}\sum_{0\neq x, \lVert x\rVert_{\infty}\leq \epsilon_0 W}\frac{f(0)}{2} +\frac{1}{\Omega_W}\sum_{ 0\neq x, \lVert x\rVert_{\infty}\leq \epsilon_0 W}\frac{f(0)}{2}e^{2\pi i \sum_{j=1}^d\frac{k_jx_j}{L}}.
		\end{align}

		
		Note that for    $\|k\|_2> \frac{L}{W}\delta$,  there is some $j$ such that \begin{equation}
			\Big|\frac{k_j}{L}\Big|>\Big|\frac{\delta}{\sqrt{d}W}\Big|.
		\end{equation} A  direct calculation shows that   
		\begin{equation}
			\sum_{0<\lVert x\rVert_{\infty}\leq  \epsilon_0  W}e^{2\pi i \sum_{j=1}^d\frac{k_jx_j}{L}}
			=\prod_{j=1}^d h\big(\frac{k_j}{L}\big)(2\lfloor \epsilon_0 W\rfloor+1)^d -1 \leq   \tilde{\rho}(W_0) (2\lfloor \epsilon_0 W\rfloor+1)^d -1, 
		\end{equation}
		where 
	\begin{equation}
		h(y)=\frac{\sin\big((2\lfloor \epsilon_0 W\rfloor+1)\pi y\big)}{(2\lfloor \epsilon_0 W\rfloor+1)\sin (\pi y)}, \quad  y\in [-\frac{1}{2}, \frac{1}{2}],
	\end{equation}
	and   the following claim  has been used:
	\\
	Given $\epsilon\in (0,1)$, there is a sufficiently large $W_0$  such that when  $W\geq W_0$
	\begin{equation} \label{upper1}
		|h(y)|\leq \tilde{\rho}(W_0)<1.
	\end{equation}
	Furthermore, for $W$ sufficiently large we see from $\Omega_W\sim W^d$ that 
	\begin{align}
		\left|\hat{p}_1(k)\right|&\leq 1-  \frac{f(0)}{2} (1-\tilde{\rho}(W_0)) 
		\frac{(2\lfloor \epsilon_0 W\rfloor+1)^d}{\Omega_W} \leq \rho(W_0,\delta) <1 
	\end{align}
	for some 
	$\rho(W_0,\delta)$.

	The verification of \eqref{upper1} is elementary.  For this, put $T=2\lfloor \epsilon_0 W\rfloor+1$, then  for  $|y|\in [\frac{1}{2T},\frac{1}{2}]$ 
	\begin{equation}
		|h(y)|\leq  \frac{1}{T|\sin(\pi y)|} \leq  \frac{1}{T|\sin(\frac{\pi}{2T} )|}  \leq \frac{1}{T_0|\sin(\frac{\pi}{2T_0} )|}<1.
	\end{equation}
	While  for  $  |y|   \in [\frac{\epsilon}{T},\frac{1}{2T}]$ 
	\begin{equation}
		|h(y)|\leq  \frac{\sin(Ty)}{T(y-\frac{1}{6}y^3)  } \leq  
		\frac{\sin(\epsilon)}{\epsilon}   \frac{1}{|1- \frac{\pi}{24T^2} |}
		\leq   \frac{\sin(\epsilon)}{\epsilon}   \frac{1}{|1- \frac{\pi}{24T_0^2} |}<1.
	\end{equation}
	Choose   $T_0$  large   and we thus complete the proof.   
\end{proof}
\subsection{Proof of Theorem \ref{pllt}}
\begin{proof}[ Proof of Theorem \ref{pllt}]
	We first focus on  the supercritical regime, and then complete the  other two cases in a unified way.
	
	To begin, it is important to observe that $\hat{p}_1(0) = 1$. Following this, we will proceed to estimate the part in which $k \neq 0$. By Lemma \ref{plf}, we have
	\begin{equation}  |\hat{p}_1(k)|^n \leq   e^{-\frac{cnW^{\alpha}}{L^{\alpha}}\|k\|_{2}^{\alpha}},\quad \forall\, \|k\|_2 \leq \delta \frac{L}{W}.
	\end{equation}
	By dividing the summation over $k$ into two parts as in Lemma \ref{plf}, we  obtain 
	\begin{equation} \label{suptwo}
		\begin{aligned}
			\left|\sum_{\|k\|_2>0}\left(\hat{p}_1(k)\right)^ne^{-2\pi i \sum_{j=1}^d\frac{k_jx_j}{L}}\right|
			&\leq  N\rho^n+
			\sum_{0<\lVert k\rVert_2\leq 
				\delta \frac{L}{W}}  e^{-cn \frac{W^{\alpha}}{L^{\alpha} }
				\|k\|_{2}^{\alpha}}.
		\end{aligned}
	\end{equation}
	The sum on the right-hand of \eqref{suptwo} can be approximated by the integral 
	\begin{equation} 
		\int_{0<\lVert k\rVert_2\leq 
			\delta \frac{L}{W}}  e^{-cn \frac{W^{\alpha}}{L^{\alpha} }
			\|k\|_{2}^{\alpha}}dk
		\leq  
		O\Big(\Big(\frac{L^{\alpha}}{nW^{\alpha}}\Big)^{1+(d-1)/ \alpha}\Big).
	\end{equation}
	When both   $n\gg L^{\alpha}/W^{\alpha}$ and $n\gg \log L$,
	\begin{equation} 
		\begin{aligned}
			\left|\sum_{\|k\|_2>0}\left(\hat{p}_1(k)\right)^ne^{-2\pi i \sum_{j=1}^d\frac{k_jx_j}{L}}\right|
			& \leq 	O\Big(\Big(\frac{L^{\alpha}}{nW^{\alpha}}\Big)^{1+(d-1)/ \alpha}\Big)
			+e^{n\log \rho+O(\log L)}=o(1),
		\end{aligned}
	\end{equation}
	from which  the proof in the  supercritical regime is  completed.
	
Next, we will establish the proofs for both the critical and subcritical regimes.
	By  Lemma \ref{plf}, we divided the sum into two parts,
	\begin{equation}
		p_n(x)=\frac{1}{N}\Big(\sum_{\|k\|_2\leq \delta\frac{L}{W}}+\sum_{\|k\|_2> \delta\frac{L}{W}}\Big)(\hat{p}_1(k))^n  e^{-2\pi i \sum_{j=1}^{d}\frac{k_jx_j}{L }}=:I_1+I_2.
	\end{equation}
	In the following  discussion, we will elucidate that $I_1 $ is the dominant term of $ p_n(x)$, whereas $ I_2 $ is relatively negligible in comparison to $I_1$. For $I_1$, noting the definition of $\alpha$-stable distribution and applying the Poisson summation formula, we have
		\begin{align}
			I_1
			&=\frac{1}{N}\sum_{\|k\|_2\leq \delta\frac{L}{W}} e^{-C_\alpha nW^{\alpha}\|k/L\|_2^{\alpha}}e^{-2\pi i \sum_{j=1}^{d}\frac{k_jx_j}{L}} \notag\\
			&=\frac{1}{N}\sum_{\|k\|_2\leq \delta\frac{L}{W}} \hat{f}_{\alpha}(n^{\frac{1}{\alpha}}W k/L)e^{-2\pi i \sum_{j=1}^{d}\frac{k_jx_j}{L }}  \notag\\
			&=\frac{1}{N}\sum_{k\in \mathbb{Z}^d}\hat{f}_{\alpha}(n^{\frac{1}{\alpha}}W k/L)e^{-2\pi i \sum_{j=1}^{d}\frac{k_jx_j}{L }}-O\Big(\frac{1}{N}\sum_{\|k\|_2\geq \delta\frac{L}{W}}e^{-C_{\alpha}nW^{\alpha}\|k/L\|_2^{\alpha}}\Big) \notag\\
			&
			=\frac{1}{n^{\frac{d}{\alpha}}W^d}
			\sum_{k \in \mathbb{Z}^d}f_{\alpha}\Big(\frac{x+k L}{n^{\frac{1}{\alpha}}W}\Big)-O\Big(\frac{1}{N}\sum_{\|k\|_2\geq \delta\frac{L}{W}}e^{-C_{\alpha}nW^{\alpha}\|k/L\|_2^{\alpha}}\Big).\label{I1LT}
		\end{align}	
	For the big O term in the last equation, we can show that 
	\begin{equation}\label{I1O}
	\frac{1}{N}	\sum_{\|k\|_2\geq \delta\frac{L}{W}}e^{-C_{\alpha}nW^{\alpha}\|k/L\|_2^{\alpha}}=\frac{1}{n^{\frac{d}{\alpha}}W^d}
		\sum_{k \in \mathbb{Z}^d}f_{\alpha}\big(\frac{x+k L}{n^{\frac{1}{\alpha}}W}\big)o(1),
	\end{equation}
	whenever $nW^{\alpha} \ll L^{\alpha}$ or $nW^{\alpha} \sim L^{\alpha}$.
By combining \eqref{I1LT} and \eqref{I1O}, we have \begin{equation}\label{I1leading}
		I_1=\frac{1}{n^{\frac{d}{\alpha}}W^d}
		\sum_{k \in \mathbb{Z}^d}f_{\alpha}\big(\frac{x+k L}{n^{\frac{1}{\alpha}}W}\big)\big(1+o(1)\big).
		\end{equation}
	
	When $nW^{\alpha}\ll L^{\alpha}$, 
	we see that  the term of $k=0$ in \eqref{I1leading} is dominant
	\begin{equation}\label{subre}
		\frac{1}{n^{\frac{d}{\alpha}}W^d}
		\sum_{k \in \mathbb{Z}^d}f_{\alpha}\Big(\frac{x+k L}{n^{\frac{1}{\alpha}}W}\Big)=\frac{1}{n^{\frac{d}{\alpha}}W^d}f_{\alpha}((nW^{\alpha})^{-1/ \alpha}x)\big(1+o(1)\big),
	\end{equation} while
	\begin{equation}
		\frac{1}{n^{\frac{d}{\alpha}}W^d}f_{\alpha}((nW^{\alpha})^{-1/ \alpha}x)=f_{\alpha}(x,nW^{\alpha}).
	\end{equation}
	When $nW^{\alpha}\sim L^{\alpha}$, we see that all terms of the summation in \eqref{I1leading} should be retained,
	\begin{equation}\label{crires}
		\frac{1}{n^{\frac{d}{\alpha}}W^d}
		\sum_{k \in \mathbb{Z}^d}f_{\alpha}\big(\frac{x+k L}{n^{\frac{1}{\alpha}}W}\big)=\frac{1}{N}\sum_{k \in \mathbb{Z}^d}f_{\alpha}(\frac{x}{L}+k,n\Big(\frac{W}{L}\Big)^{\alpha})=\frac{1}{N}\theta_{\alpha}\left(\frac{x}{L},n\Big(\frac{W}{L}\Big)^{\alpha}\right).
	\end{equation}

	For $I_2$, by Lemma \ref{plf}, we have a rough estimate,
	\begin{equation}
		|I_2|\leq \frac{1}{N}\sum_{\|k\|_2> \delta\frac{L}{W}}\rho^n .
	\end{equation}
	Further, for $|I_2|/|I_1|$, we have 
	\begin{equation}
		n^{\frac{d}{\alpha}}W^d O\Big(\|x\|_2^{d+\alpha}/\big(n^{\frac{d+\alpha}{\alpha}}W^{d+\alpha}\big)\Big)|I_2|\leq C e^{n\log \rho +\log(nW^{\alpha})+O\big(\log \|x\|_2\big)}=o(1),
	\end{equation}
	since  $n \gg \log W$  and $\log\rho<0$.  
	
	This shows that $|I_2|$  is negligible compared to the leading estimate \eqref{subre} and \eqref{crires}, so we complete the proof in the subcritical and critical  regimes.
\end{proof}
\subsection{Proof of Proposition    \ref{p-uub}}
\begin{proof}[Proof of Proposition \ref{p-uub}]
According to Theorem  \ref{pllt} and the heat kernel bound of $\alpha$-stable transition density $f_{\alpha}(x,\tau)$, when $n\gg \log L$,   there exists a constant $C > 0 $ such that 
\begin{equation}\label{uub}
p_n(x)\leq C \left(\frac{1}{n^{\frac{d}{\alpha}}W^d}\wedge \frac{nW^{\alpha}}{(\|x\|_2 \vee W) ^{d+\alpha}}+\frac{1}{N} \right).
\end{equation}
Subsequently, it is necessary to elucidate the existence of an upper bound estimate for \( n \) when it is less than $ c \log L $ for some constant $c > 0 $.

On one hand, using  the inequality
\begin{equation}  1-(1-\hat{p}_1(k))\leq e^{-(1-\hat{p}_1(k))} \quad \text{for}\quad \|k\|_2 \leq \delta \frac{L}{W},
\end{equation}
and applying Lemma \ref{plf}, we can divide $p_n$ into two parts
\begin{align}\label{abpn}
p_n(x)&=\frac{1}{N}\sum_{k\in \Lambda_{L,d}}(\hat{p}_1(k))^ne^{-2\pi i \sum_{j=1}^{d}\frac{k_jx_j}{L }}\notag\\
&\leq \frac{1}{N}\Big(\sum_{\|k\|_2\leq \delta \frac{L}{W}}e^{-n(1-\hat{p}_1(k))}+\sum_{\|k\|_2> \delta\frac{L}{W}}\rho^{n-2}(\hat{p}_1(k))^2\Big).
\end{align}
For the first term in the above expression, since $f_{\alpha}$ is a probability density function,  the application of the Poisson summation formula yields \begin{align}\label{firboud}
\frac{1}{N}\sum_{\|k\|_2<\eta \frac{L}{W}}e^{-n(1-\hat{p}_1(k))}&\leq 	\frac{1}{N}\sum_{k \in \mathbb{Z}^d}\hat{f}_{\alpha}\Big(\frac{n^{\frac{1}{\alpha}}Wk}{L}\Big)\notag\\
&=\frac{1}{n^{\frac{d}{\alpha}}W^d}\sum_{x \in \mathbb{Z}^d}
f_{\alpha}\Big(\frac{xL}{n^{\frac{1}{\alpha}}W}\Big)\notag\\
&\leq \frac{O(1)}{n^{\frac{d}{\alpha}}W^d}.
\end{align}
For the second term in \eqref{abpn}, we observe that 
\begin{equation}
\frac{1}{N}\sum_{k \in \Lambda_{L,d}}(\hat{p}_1(k))^2=p_2(0)\leq \|p_1\|_{\infty},
\end{equation} 
which leads to 
\begin{equation}\label{secboud}
\frac{1}{N}\sum_{\|k\|_2> \delta\frac{L}{W}}\rho^{n-2}(\hat{p}_1(k))^2 \leq \rho^{n-2}\|p_1\|_{\infty}\leq O\Big(\frac{1}{W^d}\Big)n^{-\frac{d}{\alpha}}.
\end{equation}
Taking into account the results from the first part \eqref{firboud} and the second part \eqref{secboud},   we can conclude that there exists a  positive constant $C_1$ such that
\begin{equation}\label{abpnb}
p_n(x)\leq \frac{C_1}{n^{\frac{d}{\alpha}}W^d}.
\end{equation}

On the other hand, we claim that the $n$-steps transition probability $p_n(x)$ satisfies 
\begin{equation}\label{cla}
p_n(x) \leq C_2\frac{ nW^{\alpha}}{(\|x\|_2 \vee W)^{d+\alpha}},
\end{equation}
where $C_2$ is a positive constant.
If this holds true, by combining \eqref{abpnb} and \eqref{cla}, we derive the  upper bound estimate 
\begin{equation}\label{plnsmab}
p_n(x)\leq C_3 \left(\frac{1}{n^{\frac{d}{\alpha}}W^d}\wedge \frac{nW^{\alpha}}{(\|x\|_2 \vee W) ^{d+\alpha}} \right).
\end{equation}
Having established \eqref{uub} and \eqref{plnsmab}, we can therefore complete the proof of Proposition \ref{p-uub}. 

The remaining task is to prove the above claim \eqref{cla}.
Firstly, within the setting of spread-out power law models,  it's easy to see from the definition of $p_1$ in \eqref{defp1} that
\begin{equation}\label{1sbound}
p_1(x) \leq O(W^{\alpha})(\|x\|_2\vee W)^{-d-\alpha}.
\end{equation} 
In the following analysis, we define the event 
\begin{align}
&	A=\{\text{random walks in which all steps are shorter than $C_4(\|x\|_2 \vee W)$ }\},
\end{align}
where $C_4 > 0$ is a constant.
This framework allows us to bound $ p_n(x)$ by considering the contributions from both events $ A $ and the complement   $A^{c}$.

By applying the local Central Limit Theorem on $\Lambda_L^d$ to the event $A$, and  employing the bound for $p_1$ provided in \eqref{1sbound} to $A^{c}$, we derive the following bound for $p_n$
\begin{align}
	p_n(x)&=\mathbb{P}(A)+\mathbb{P}(A^{c})  \notag
	\leq (\tilde{\sigma} n)^{-\frac{d}{2}}e^{-\frac{\|x\|^2}{\tilde{\sigma}n}} + n p_1(y)\\
   &\leq \frac{O(\tilde{\sigma}n)}{(\|x\|_2 \vee W)^{d+2}}+\frac{O(nW^{\alpha})}{(\|x\|_2 \vee W)^{d+\alpha}}, \label{lcltb}
\end{align}
where \(\|y\|_2 \geq C_4(\|x\|_2 \vee W)\) and \(\tilde{\sigma}\) denotes the variance of the truncated one-step transition probability  defined as
 \begin{equation}
 	\tilde{p}_1(z) := p_1(z) \mathbb{I}\{\|z\|_2 < C_4(\|x\|_2 \vee W)\}.
 \end{equation} This variance can be expressed as  
\begin{equation}\label{varb}
	\tilde{\sigma} = \sum_{z \in \Lambda_L^d} \|z\|_2^2 \tilde{p}_1(z) \leq O(W^{\alpha})(\|x\| \vee W)^{2-\alpha}.
\end{equation}
By substituting the upper bound  of $\tilde{\sigma}$ into \eqref{lcltb}, we obtain 
\begin{equation}
	p_n(x) \leq \frac{O(nW^{\alpha})}{(\|x\|_2 \vee W)^{d+\alpha}}.
\end{equation}
Thus, we complete the proof of the claim  \eqref{cla}.
\end{proof}

\section{Proof of Theorem \ref{LLT}}  \label{sect3}

Similarly, for Theorem \ref{LLT},  we need to use  the (inverse) discrete Fourier transform  to obtain    \eqref{IFT}. 
A key estimate  of $\hat{p_1}$  can be  summarized as follows.
\begin{lemma}\label{lem2parts}
For  $\delta>0$ sufficiently small, we have  
\begin{equation}
\left|\hat{p}_1(k)\right|\leq 1- \pi^2  \sigma_{L}^2,  \quad \forall\, \|k\|_2 \leq \delta \frac{L}{W}
\end{equation}
and 
there exists $\rho<1$
such that \begin{equation}
\left|\hat{p}_1(k)\right|\leq \rho, \quad \forall\,\|k\|_2>\delta \frac{L}{W}.
\end{equation}
\end{lemma}
\begin{proof}
Since $f$ is   symmetric,  we have  $\sum_{x} x_j p_1(x)=0$  for all $j$. With the notation \eqref{LVar} in mind,  it is easy to see from 
the inequality 
\begin{equation}
\left|e^{iy}-\sum_{m=0}^2\frac{(iy)^m}{m!}\right|\leq \min\left\{\frac{1}{6}|y|^{3},|y|^2\right\}, \quad  \forall y\in \mathbb{R}
\end{equation}
we have
\begin{align} \label{pTaylor2}
&\left| \hat{p}_1(k)-\left(1-2\pi^2  \sigma_{L}^2 \right)\right| \nonumber\\
&\leq \sum_{0\neq x \in \Lambda_{L,d}}\frac{(2\pi)^2}{\Omega_W}f\left(\frac{x}{W}\right)
\min\left\{\frac{\pi}{3}\bigg|\sum_{j=1}^d\frac{k_jx_j}{L}\bigg|^{3},\bigg|\sum_{j=1}^d\frac{k_jx_j}{L}\bigg|^2\right\}.
\end{align}
Thus, we can choose  a sufficiently  small   constant $\delta>0$   such that  the right-hand side of \eqref{pTaylor2}  is less than  $\pi^2  \sigma_{L}^2$ whenever $\|k\|_2\leq \frac{L}{W}\delta$.  This shows for  $\|k\|_2\leq \frac{L}{W}\delta$ that 
\begin{equation}
\left|\hat{p}_1(k)\right|\leq 1- \pi^2  \sigma_{L}^2. 
\end{equation}

The second part of Lemma \ref{lem2parts} is similar to that of Lemma \ref{plf}, whose proof is independent of the density function. Thus, we complete the proof of the lemma.
\end{proof}
Near the origin 
of  $k$ we have a more exact  approximation for $\hat{p}_n(k)$.

\begin{lemma}\label{Taylorlocal} Assume that 
$\int_{\mathbb{R}^d} \|x\|_{2}^3  f(x)dx<\infty$, for any given  $\delta_0\in (0,\frac{1}{6})$  we have 
\begin{equation}\hat{p}_n(k)	=e^{-2\pi^2  n\sigma_{L}^2} \Big(1+O\big(n^{-\frac{1}{2}+3\delta_0}\big)\Big)\end{equation}
uniformly for \begin{equation}\|k\|_2 \leq n^{-\frac{1}{2}+\delta_0}\frac{L}{W}.\end{equation}
\end{lemma}
\begin{proof}
By the   Taylor expansion, 
\begin{align} \label{pTaylor3}		
\hat{p}_1(k)=1-2\pi^2  \sigma_{L}^2+ 
O\Big( \Big|\sum_{j=1}^d\frac{k_j}{L}\Big|
^{3}\Big).
\end{align}
taking the $n$-th power on both sides, the desired result immediately follows.  \end{proof}


Now we are ready to complete the proof of Theorem \ref{LLT}.
\begin{proof}[Proof of Theorem \ref{LLT}]


	
	We first focus on  the supercritical regime, and then complete the  other two cases in a unified way.

	Note that the covariance matrix $\Gamma$ of $f$ is positive definite,  it  is easy to see from \eqref{LVar} that there is a positive constant $c$ such that  
	\begin{equation} \label{p1k-001} \sigma_{L}^2\geq   \frac{cW^2}{2\pi^2L^2}\|k\|_{2}^2,
	\end{equation}
	from which 
	\begin{equation} \label{p1k-1} |\hat{p}_1(k)|^n \leq   e^{-\frac{cnW^2}{L^2}\|k\|_{2}^2},\quad \forall\, \|k\|_2 \leq \delta \frac{L}{W}.
	\end{equation}
	By dividing the summation over $k$ into two parts as in Lemma \ref{lem2parts}, we  obtain 
	\begin{equation} \label{supbound1}
		\begin{aligned}
			\left|\sum_{\|k\|_2>0}\left(\hat{p}_1(k)\right)^ne^{-2\pi i \sum_{j=1}^d\frac{k_jx_j}{L}}\right|
			&\leq  N\rho^n+
			\sum_{0<\lVert k\rVert_2\leq 
				\delta \frac{L}{W}}  e^{-cn \frac{W^2}{L^2 }
				\|k\|_{2}^2}.
		\end{aligned}
	\end{equation}
	For $k\in \mathbb{Z}^d$, it's easy to see that  $\lVert k\rVert_{2}^2\geq \sum_{j=1}^d |k_j|$ and when $\lVert k\rVert_2>0$ there is some $j$ such that $|k_j|\geq 1$, so we  get 
	\begin{equation} \label{supbound2}
		\sum_{0<\lVert k\rVert_2\leq 
			\delta \frac{L}{W}}  e^{-cn \frac{W^2}{L^2 }
			\|k\|_{2}^2}
		\leq   2^d \Big(1-e^{-cn \frac{W^2}{L^2}}\Big)^{-d} 
		e^{-cn \frac{W^2}{L^2}}.
	\end{equation}
	When both   $n\gg L^2/W^2$ and $n\gg \log L$  and further 
	\begin{equation} \label{supbound3}
		\begin{aligned}
			\left|\sum_{\|k\|_2>0}\left(\hat{p}_1(k)\right)^ne^{-2\pi i \sum_{j=1}^d\frac{k_jx_j}{L}}\right|
			& \leq O\Big(e^{-cn \frac{W^2}{L^2}}\Big)
			+e^{n\log \rho+O(\log L)}=o(1),
		\end{aligned}
	\end{equation}
	from which   the proof in the  supercritical regime is  completed. 
	
	Next, we turn to the subcritical and critical regimes.  First,  rewrite $p_n(x)$ in $\eqref{IFT}$ as a sum of three parts 
\begin{align}
p_n(x)
	&=\frac{1}{N}\bigg(\sum_{\|k\|_2<\eta \frac{L}{W}}+\sum_{\eta \frac{L}{W}\leq\|k\|_2\leq \delta \frac{L}{W}}+\sum_{\|k\|_2>\delta \frac{L}{W}} \bigg)(\hat{p}_1(k))^n  e^{-2\pi i \sum_{j=1}^{d}\frac{k_jx_j}{L }}\\
	&=:\frac{1}{N}\left(\Sigma_1+\Sigma_2+\Sigma_3\right)   \label{sigma123},
	\end{align}
where $\eta=n^{-\frac{1}{2}+\delta_0}$ with $0<\delta_0<1/6$.

For  $\Sigma_1$, by Lemma \ref{Taylorlocal}, 
we have 
\begin{align} \label{sigma-1}
\Sigma_1&= \big(1+o(n\eta^3)\big)
\sum_{\|k\|_2<\eta \frac{L}{W}} 
\exp\{ -2\pi^2n W^2/L^2<k\Gamma_{L},k>-2\pi  i <x/L,k>\}\nonumber\\
&=\big(1+o(n\eta^3)\big) \bigg(
\theta\big(x/L,  2\pi n W^2/L^2\Gamma_{L}\big) 
-O\Big(\sum_{\|k\|_2\geq \eta \frac{L}{W}} \exp\{-2\pi^2n\sigma_{L}^2\}\Big) \bigg).
\end{align}
To obtain the bound  of the big O term  in  \eqref{sigma-1},  keep the first few terms and bound the remaining sum  by an integral  because of the  monotonically decreasing     function,     by \eqref{p1k-001},  we have 
\begin{align}
&\sum_{\|k\|_2\geq \eta \frac{L}{W}} \exp\{-2\pi^2n\sigma_{L}^2\}\leq   \sum_{\|k\|_2\geq \eta \frac{L}{W}} e^{-cn \frac{W^2}{L^2 }\|k\|_{2}^2}\nonumber\\
& \leq  \sum_{\eta \frac{L}{W}\leq \|k\|_2\leq \eta \frac{L}{W}+\sqrt{d}}  e^{-cn \frac{W^2}{L^2} \|k\|_{2}^2}+
\int_{\|k\|_2\geq \eta \frac{L}{W}}  e^{-cn \frac{W^2}{L^2} \|k\|_{2}^2}dk \nonumber \\
& \leq C \Big(\big(\eta\sqrt{n}+ \frac{\sqrt{d n}W}{L}\big)^d +1\Big) \Big(\frac{L}{\sqrt{n}W} \Big)^d e^{-c\eta^2n},
\end{align}
for some $C>0$. Since  $1\ll \eta \sqrt{n}\leq \eta L/W$,   for   $nW^2 \ll L^2$  or $nW^2 \sim L^2$  we can  show that  
\begin{align} \label{sigma1-001}
\sum_{\|k\|_2\geq \eta \frac{L}{W}} \exp\{-2\pi^2n\sigma_{L}^2\}  =\theta\big(x/L,  2\pi n W^2/L^2 \Gamma_{L}\big)o(1),
\end{align}
whenever  $\|x\|_2 \ll n^{2/3}W$. 
With \eqref{sigma-1} and \eqref{sigma1-001} in mind, we obtain \begin{equation}\label{sigmacri}
\Sigma_1=\big(1+o(n\eta^3)\big) \theta\big(x/L,  2\pi n W^2/L^2 \Gamma_{L}\big).
\end{equation}

The well-known reciprocity formula\cite{bl61}  gives us
\begin{align} 
&\theta\big(x/L,  2\pi n W^2/L^2 \Gamma_{L}\big)=\notag\\
&\frac{N}{ (\sqrt{2\pi nW^2\det \Gamma_L})^{d}}  \sum_{k\in \mathbb{Z}^d} \exp \{-\frac{1}{2nW^2} 
<(kL+x)\Gamma_L^{-1},(kL+x)>  \}.\label{dualeq}
\end{align}
When $nW^2 \ll L^2$,  noting that $|x_j/L|\leq 1/2$ for all $j$ and using lemma \ref{lem2parts},  we see that  the term of $k=0$ in \eqref{dualeq} is dominant and   thus  
\begin{align} \label{dualityeq1}
&\theta\big(x/L,  2\pi n W^2/L^2 \Gamma_{L}\big)= 
\frac{N}{ (\sqrt{2\pi nW^2\det \Gamma_L})^{d}} e^{-\frac{1}{2nW^2} 
	<x\Gamma_L^{-1},x>}   \Big( 1+o(1)\Big).
\end{align}
When $nW^2 \sim L^2$, we can see that all terms of the summation in \eqref{dualeq} should be retained.

For $\Sigma_2$ and $\Sigma_3$ in \eqref{sigma123}, whenever  $\|x\|_2 \ll n^{2/3}W$ and $n\gg \log W$,  by Lemma \ref{lem2parts},   

\begin{equation} \label{sigma2}
\Big(\frac{\sqrt{n}W}{L} \Big)^d 
e^{O\big(\frac{\|x\|_2^2}{nW^2}\big)}	\Big|\Sigma_2\Big| \leq C
e^{-cn\eta^2+d\log(\sqrt{nW})+O\big(\frac{\|x\|_2^2}{nW^2}\big)}=o(1),
\end{equation}
while 
\begin{equation} \label{sigma3}
\Big(\frac{\sqrt{n}W}{L} \Big)^d 
e^{O\big(\frac{\|x\|_2^2}{nW^2}\big)}	\Big|\Sigma_3\Big| \leq C
e^{n\log \rho +d\log(\sqrt{nW})+O\big(\frac{\|x\|_2^2}{nW^2}\big)}=o(1),
\end{equation}
since  $\log\rho<0$.  This shows that they  are negligible compared to the leading term $\Sigma_1$.

Combining 
\eqref{sigmacri}, \eqref{dualityeq1}, \eqref{sigma2} and \eqref{sigma3}, we thus   complete the proof in the subcritical and critical  regimes.
\end{proof}

\begin{acknow}
We would like to thank Xin Chen and Guangyi Zou for their valuable discussions.
This work was supported by the National Natural Science Foundation of China \#12371157
and \#12090012.

\end{acknow}

\end{document}